\documentclass[10pt,a4paper]{article}
\usepackage{amsmath}
\usepackage{amsfonts}
\usepackage{amsthm}
\usepackage{relsize}

\newtheorem{thm}{Theorem}[section]
\newtheorem{lem}[thm]{Lemma}
\newtheorem{prop}[thm]{Proposition}
\newtheorem{cor}[thm]{Corollary}

\newtheorem*{thma}{Theorem A}
\newtheorem*{thmb}{Theorem B}
\newtheorem*{thmc}{Theorem C}
\newtheorem*{thmd}{Theorem D}
\newtheorem*{thme}{Theorem E}
\newtheorem*{corn}{Corollary}
\newtheorem*{propn}{Proposition}

\newcommand{\db}{\mathrm{db}}
\newcommand{\eb}{\mathrm{eb}}
\newcommand{\Aut}{\mathrm{Aut}}
\newcommand{\Alt}{\mathrm{Alt}}
\newcommand{\FD}{\mathrm{[FD]}}
\newcommand{\FR}{\mathrm{[FR]}}

\newcommand{\mcS}{\mathcal{S}}
\newcommand{\mcX}{\mathcal{X}}

\newcommand{\bF}{\mathbb{F}}

\newcommand{\bP}{\mathbb{P}}

\begin{document}

\title{The generalised Fitting subgroup of a profinite group}

\author{Colin Reid\\
School of Mathematical Sciences\\
Queen Mary, University of London\\
Mile End Road, London E1 4NS\\
c.reid@qmul.ac.uk}

\maketitle

\begin{abstract}The generalised Fitting subgroup of a finite group is the group generated by all subnormal subgroups that are either nilpotent or quasisimple.  The importance of this subgroup in finite group theory stems from the fact that it always contains its own centraliser, so that any finite group is an abelian extension of a group of automorphisms of its generalised Fitting subgroup.  We define a class of profinite groups which generalises this phenomenon, and explore some consequences for the structure of profinite groups.\end{abstract}


\paragraph{NB} Throughout this paper, we will be concerned with profinite groups as topological groups.  As such, it will be tacitly assumed that groups have a natural topology (depending on their construction), that subgroups are required to be closed, that homomorphisms are required to be continuous, and that generation refers to topological generation.  When we wish to suppress topological considerations, the word `abstract' will be used, for instance `abstract subgroup'.

\section{Introduction}

Thanks to Sylow's theorems, the theory of finite groups can be broken to a large extent into two parts: firstly the study of groups of prime power order, and secondly the study of interactions within a finite group between subgroups of prime power order and elements of order coprime to them.  This description holds especially for finite soluble groups, in light of Hall's theorem, and also the fact that any finite soluble group is determined up to an abelian normal subgroup by its action on the Fitting subgroup, which is itself a direct product of groups of prime power order.  For the general finite case, the analysis often reduces to studying finite soluble groups, together with a study of finite groups which come close to being simple, namely, those groups $G$ containing a normal subgroup $H$ such that $C_G(H) = Z(H)$ and $H/Z(H)$ is simple.\\

In several respects, profinite groups, regarded as topological groups, have a structure which is exactly analogous to the structure of finite groups.  The equivalent of the $p$-subgroups of a finite group turn out to be the (closed) pro-$p$ subgroups of a profinite group, and a version of Sylow's theorems holds for these.  This suggests an approach to the study of profinite groups which is analogous to the approach found in finite group theory: first, study pro-$p$ groups, and second, study coprime interactions invoving pro-$p$ subgroups whose internal structure is assumed to be understood.  In this paper, we are concerned with the second part of this approach.

\paragraph{Notation} We will denote by $\bP$ the set of prime numbers, $\pi$ some subset thereof, and $\pi'$ its complement in $\bP$.  When a group is referred to as pro-$\mcX$ for some property $\mcX$, this is understood to mean that the group is the inverse limit of finite $\mcX$-groups.  Given a topological group $G$, we indicate the closure of the derived group by $G'$, and $G^{(n)}$ is defined inductively by $G^{(0)}=G$ and $G^{(n)} = (G^{(n-1)})'$; on the other hand, $G^n$ indicates the subgroup topologically generated by $n$-th powers of elements of $G$.\\

Given the role of the generalised Fitting subgroup in finite group theory, and the role of pro-$p$ groups in profinite group theory, we are motivated to make the following definitions:

\paragraph{Definitions}  Let $G$ be a profinite group.

The \emph{$\pi$-core} $O_\pi(G)$ of $G$ is the group generated by all normal pro-$\pi$ subgroups of $G$.  The \emph{Fitting subgroup} $F(G)$ of $G$ is the subgroup generated by the subgroups $O_p(G)$, as $p$ ranges over the primes.

We say a profinite group $Q$ is \emph{quasisimple} if $Q$ is perfect and $Q/Z(Q)$ is simple.  We say $Q$ is a \emph{component} of $G$ if it is a subnormal quasisimple subgroup of $G$.  The \emph{layer} $E(G)$ of $G$ is the subgroup generated by all components of $G$.  The \emph{generalised Fitting subgroup} $F^*(G)$ of $G$ is the subgroup generated by $F(G)$ and $E(G)$.\\

Internally, the generalised Fitting subgroup of a profinite group has the structure we expect: it is a central product of pro-$p$ groups and quasisimple groups.  It also contains every subnormal pronilpotent subgroup of $G$.  However, a given non-trivial profinite group $G$ may not have any non-trivial pronilpotent or quasisimple subnormal subgroups.

\paragraph{Definitions}  Let $G$ be a profinite group.  Say $G$ is \emph{Fitting-degenerate} if $F^*(G)$ is trivial.  Say $G$ is \emph{Fitting-regular} if no non-trivial image of $G$ is Fitting-degenerate.

We write $\FD$ for the class of Fitting-degenerate profinite groups and $\FR$ for the class of Fitting-regular profinite groups.  Note that $\FD \cap \FR$ is the class of trivial groups, and that $\FR$ contains the class of finite groups, since a minimal normal subgroup of a finite group $G$ will always be either abelian or contained in $E(G)$ (or both).\\

The author is not aware of any definitions similar to these in the existing literature.  The goal of this paper is therefore to serve as an introduction to these properties and their role in the structure of profinite groups.\\

All profinite groups are constructed from Fitting-regular and Fitting-degen\-erate groups in the following manner, giving a decomposition which also behaves well with respect to open subgroups.

\begin{thma}Let $G$ be a profinite group.  Then $G$ has a characteristic subgroup $R=O_\FR(G)$, called either the \emph{Fitting-regular radical} or the \emph{Fitting-degenerate residual}, such that:

(i) A subnormal subgroup of $G$ is Fitting-regular if and only if it is contained in $R$;

(ii) $G/R$ is Fitting-degenerate, and covers every Fitting-degenerate quotient of $G$.\\

We also have $O_\FR(H) = O_\FR(G) \cap H$, given any open subgroup $H$ of $G$.
\end{thma}

The most important property of the generalised Fitting subgroup in finite group theory is that it always contains its own centraliser.  Given the existence of non-trivial Fitting-degenerate profinite groups, this property does not generalise to all profinite groups.  It does however generalise to Fitting-regular profinite groups:

\begin{thmb}Let $G \in \FR$.  Then $C_G (F^*(G)) = Z(F(G))$.\end{thmb}

We now define some invariants in order to describe sufficient conditions for a profinite group to be Fitting-regular.

\paragraph{Definitions}Given a profinite group $G$, let $d_p(G)$ denote the size of a minimum generating set for a $p$-Sylow subgroup of $G$; this is also equal to the rank of $S/\Phi(S)$, where $S$ is any $p$-Sylow subgroup of $G$.  We write $d_*(G)$ for the supremum of $d_p(G)$ as $p$ ranges over all primes.\\

Given a finitely generated pro-$p$ group $S$, the automorphism group of $S$ acts on the elementary abelian group $S/\Phi(S)$.  However, it need not act irreducibly.  In general, $S/\Phi(S)$ will have a normal series in which each factor is irreducible as an $\Aut(S)$-module, and by the Jordan-H\"{o}lder theorem, the rank of the factors is uniquely determined.  Let $c(S)$ denote the maximum rank which occurs.\\

If $G$ has a finitely generated $p$-Sylow subgroup $S$, define $c_p(G) := c(S)$; otherwise, define $c_p(G) := d_p(G)$.  Note that by definition, $c_p(G) \leq d_p(G)$ for any $G$.  Write $c_*(G)$ for the supremum of $c_p(G)$ as $p$ ranges over all primes.

\paragraph{Definition}Let $G$ be a profinite group.  Say $G$ is \emph{Sylow-finite} if every Sylow subgroup of $G$ is contained in a finite normal subgroup.\\

By Dicman's lemma (see \S 2), this condition is equivalent to requiring each Sylow subgroup to be finite and to have a centraliser in $G$ of finite index.\\

Given a profinite group $G$ for which $c_*(G)$ is finite, we obtain the following description:

\begin{thmc}Let $G$ be a profinite group such that $c_*(G)$ is finite.  Then $G$ is virtually pronilpotent-by-(Sylow-finite)-by-abelian.\end{thmc}

\begin{corn}Let $G$ be a profinite group such that $c_*(G)$ is finite, and let $H$ be a subgroup of $G$.  Then $H$ is Fitting-regular.\end{corn}

We also note that some alternative restrictions on Sylow subgroups do not enforce Fitting-regularity.

\begin{propn}(i) There exist non-trivial Fitting-degenerate prosoluble groups, all of whose Sylow subgroups are finite elementary abelian groups.\\
 
(ii) There exist non-trivial Fitting-degenerate prosoluble groups which are countably based pro-$\{p,q\}$ groups, for any distinct primes $p$ and $q$.\end{propn}

\paragraph{Remark}There are also examples in the literature of non-trivial finitely generated profinite groups which are Fitting-degenerate.  Indeed, Lucchini (\cite{Luc}) has constructed a just infinite $2$-generated profinite group which can easily be seen to be Fitting-degenerate.\\

Fitting-degeneracy can also be characterised in terms of finite images, even though the finite images are themselves automatically Fitting-regular.

\begin{thmd}Let $G$ be a non-trivial profinite group.  Then $G \in \FD$ if and only if the following holds, for any $x \in G \setminus 1$:
 
(*) There is an open normal subgroup $K$ of $G$, depending on $x$, such that $G/K$ is primitive and $xK$ is not contained in $F^*(G/K)$.\end{thmd}

Finally, we note that to determine whether or not a group is Fitting-regular, it suffices to consider countably based images:

\begin{thme}Let $G$ be a profinite group that is not Fitting-regular.  Then there is a normal subgroup $N$ of $G$ such that $G/N$ is countably based and Fitting-degenerate.\end{thme}

\section{Preliminaries}

We recall the theorems of Zassenhaus and Mal'cev on soluble linear groups.

\begin{thm}[Zassenhaus \cite{Zas}]
Let $G$ be a soluble linear group of degree $n$ over any field.  Then the derived length of $G$ is bounded by a function of $d$.\end{thm}

Given a finite-dimensional vector space $V$, we say that a subgroup of the general linear group $GL(V)$ of $V$ is \emph{triangularisable} if it is conjugate to a subgroup of the group $Tr(V)$ of upper-triangular matrices with respect to some basis.  Note that the derived subgroup of $Tr(\bF^n_p)$ is a $p$-group for any integer $n$, and hence any triangularisable group over a finite-dimensional $\bF_p$-vector space has this property.  We will write $GL(n,p)$ for $GL(\bF^n_p)$.

\begin{thm}[Mal'cev \cite{Mal}]Let $G$ be a soluble linear group of degree $n$ over any field.  Then $G$ has a triangularisable normal subgroup $T$, such that $|G:T|$ is bounded by a function of $n$.\end{thm}

\begin{cor}\label{malcor}Let $G$ be a finite soluble subgroup of $GL(n,p)$.  Then there are functions $\eb(n)$ and $\db(n)$ which do not depend on $p$, such that $(G^{\eb(n)})'$ is a $p$-group and $G^{(\db(n))}$ is trivial.\end{cor}

\begin{proof}Zassenhaus's theorem implies the existence of a suitable function $\db$.  By Mal'cev's theorem, $G$ has a triangularisable normal subgroup $T$ of index bounded by a function of $n$ alone.  It follows that we can find $\eb(n)$ such that $G^{\eb(n)}$ is necessarily a subgroup of $T$, by considering the exponent of $G/T$.  It follows that $G^{\eb(n)}$ is triangularisable, so its derived group is a $p$-group.\end{proof}

The functions $\eb$ and $\db$ will be used later, though for the purposes of this paper we do not need to know their rate of growth, merely their existence.\\

We will also make use of Dicman's lemma, a general theorem in the theory of infinite groups, in the proofs of Theorems A and C.

\begin{lem}[Dicman \cite{Dic}]Let $G$ be an abstract group, and let $H$ be a subgroup generated by a finite set of elements $X$.  Suppose that each element of $X$ is of finite order and has finitely many conjugates in $G$.  Then $H$ is finite.\end{lem}

In addition, we need some standard facts from profinite group theory, which are direct analogues of finite results.

\paragraph{Definition} We define a \emph{supernatural number} to be a formal product $\prod_{p \in \bP} p^{n_p}$ of prime powers, where each $n_p$ is either a non-negative integer or $\infty$.\\

Multiplication of supernatural numbers is defined in the obvious manner; note that any set of supernatural numbers has a supernatural least common multiple.  Also, by unique factorisation, the set of supernatural numbers contains a copy of the multiplicative semigroup of natural numbers, which we may regard as the \emph{finite} supernatural numbers.  A \emph{$\pi$-number} is a supernatural number $\prod_{p \in \bP} p^{n_p}$ for which $n_p=0$ for all $p$ in $\pi'$.

\paragraph{Definition} Let $G$ be a profinite group and let $H$ be a subgroup.  We define the \emph{index} $|G:H|$ of $H$ in $G$ to be the least common multiple of $|G/N:HN/N|$ as $N$ ranges over all open normal subgroups of $G$, and the \emph{order} of $G$ to be $|G:1|$.  In particular, the order of a profinite group is a $\pi$-number if and only if the group is pro-$\pi$.  (Note that in contrast to finite group theory, the supernatural order of a profinite group is not determined by the cardinality of the underlying set.)  Given an element $x$ of $G$, we define the order of $x$ to be $|\langle x \rangle : 1|$.

\begin{thm}Let $G$ be a profinite group, and let $H$ be a subgroup.  Then $|G|$ factorises as a supernatural number as $|G:H||H|$.  If $H$ is normal then $|G:H|=|G/H|$.\end{thm}

\begin{proof}See \cite{Wil}.\end{proof}

\paragraph{Definition} Let $G$ be a profinite group, and let $H$ be a subgroup.  We say $H$ is a \emph{$\pi$-Hall subgroup} of $G$ if $H$ is a pro-$\pi$ group, and $|G:H|$ is a $\pi'$-number.  We also refer to $\{p\}$-Hall subgroups as \emph{$p$-Sylow subgroups}.

\begin{thm}[Sylow's and Hall's theorems for profinite groups]
Let $G$ be a profinite group.  If $G$ is prosoluble, let $\pi$ be an arbitrary set of primes; otherwise, let $\pi$ consist of a single prime.\\

(i) $G$ has a $\pi$-Hall subgroup.\\

(ii) Any two $\pi$-Hall subgroups of $G$ are conjugate.\\

(iii) Every pro-$\pi$ subgroup of $G$ is contained in some $\pi$-Hall subgroup.\end{thm}

\begin{proof}See \cite{Wil}.\end{proof}

\begin{cor}Let $G$ be a profinite group, and let $H$ be a subgroup of $G$.  Let $\mcS$ be a set of subsets of the prime numbers such that $\bigcup \mcS = \mu$, and suppose $H$ contains a $\pi$-Hall subgroup of $G$ for every $\pi \in \mcS$.  Suppose also that either $G$ is prosoluble or that $\mu = \bP$.  Then $H$ contains a $\mu$-Hall subgroup of $G$.\end{cor}

Here is a standard result about automorphisms of pro-$p$ groups.

\begin{prop}Let $S$ be a countably based pro-$p$ group, and let $H$ be a profinite group of automorphisms of $S$.  Let $(S_1,S_2,\dots)$ be a descending sequence of $H$-invariant normal subgroups of $S$, such that the intersection of the $S_i$ is exactly $\Phi(S)$, and suppose $H$ acts as a pro-$p$ group on $S_i/S_{i+1}$ for every $i$.  Then $H$ is a pro-$p$ group.\end{prop}

\begin{proof}See \cite{DDMS}.\end{proof}

\begin{cor}Let $S$ be a finitely generated pro-$p$ group with $c(S) = c$.  Let $H$ be a prosoluble group of automorphisms of $S$.  Let $K = H^{\eb(c)}H^{(\db(c))}$.  Then $K'$ is a pro-$p$ group.\end{cor}

\begin{proof}Consider the action of $H$ on a normal series for $S/\Phi(S)$ with maximum rank $c$.  On each factor, $H$ must act as a subgroup of $GL(c,p)$, which is necessarily finite and soluble since $H$ is prosoluble.  By Corollary \ref{malcor}, $K'$ must act on each factor as a $p$-group.  Hence $K'$ is pro-$p$ by the proposition.\end{proof}

\paragraph{Definition}The \emph{Frattini subgroup} $\Phi(G)$ of a profinite group $G$ is the intersection of all maximal open subgroups of $G$.\\

The role of the Frattini subgroup in finite group theory generalises to profinite groups, and the class of pronilpotent groups can be characterised in terms of its Sylow structure in a similar manner to the class of finite nilpotent groups.

\begin{lem}(i) Let $G$ be a profinite group.  Then $G$ is pronilpotent if and only if it is the Cartesian product of its Sylow subgroups, or equivalently, if and only if every Sylow subgroup is normal.\\

(ii) Let $G$ be a profinite group, and let $K$ be a normal subgroup of $G$ containing $\Phi(G)$.  Then $K$ is pronilpotent if and only if $K/\Phi(G)$ is pronilpotent.  In particular, $\Phi(G)$ is pronilpotent.\\

(iii) Let $G$ be a profinite group.  If $X$ is a set of elements of $G$, then $X$ generates $G$ if and only if the image of $X$ in $G/\Phi(G)$ generates $G/\Phi(G)$.\\
 
(iv) Let $S$ be a pro-$p$ group.  Then $S/\Phi(S)$ is the largest elementary abelian image of $S$, and the cardinality of a minimum generating set of $S$ is exactly the rank of $S/\Phi(S)$.\end{lem}

\begin{proof}See \cite{Wil}.\end{proof}

Given a profinite group $G$ and an integer $r$, we define the \emph{lower $r$-series} $\Phi^k_r=\Phi^k_r(G)$ of $G$ as follows: $\Phi^0_r=G$, and thereafter $\Phi^{k+1}_r$ is the smallest subgroup such that $\Phi^k_r/\Phi^{k+1}_r$ is a central section of $G$ of exponent dividing $r$.  We also define $\Phi^\infty_r$ to be $\bigcap_{i \geq 0} \Phi^k_r$; by considering finite images, we see that $G/\Phi^\infty_r$ is the largest pronilpotent image of $G$ to involve only the primes dividing $r$.  Note that for $p$ a prime and $S$ a pro-$p$ group, we have $\Phi^1_p(S)=\Phi(S)$ and $\Phi^\infty_p(S)=1$.\\

The following theorem was given by Tate in the finite case, but the generalisation from finite to profinite groups is immediate.

\begin{thm}[Tate \cite{Tat}]
Let $G$ be a profinite group and let $r$ be an integer.  Suppose $H$ is a subgroup of $G$ such that $|G:H|$ is coprime to $r$ and $\Phi^1_r(H) = \Phi^1_r(G) \cap H$.  Then $\Phi^k_r(H) = \Phi^k_r(G) \cap H$, for $k$ any positive integer and also for $k=\infty$.\end{thm}

This theorem has important consequences for fusion in profinite groups, given a finitely generated Sylow subgroup.  Foremost of these are the following:

\begin{cor}Let $G$ be a profinite group, such that $G$ has a finitely generated $p$-Sylow subgroup $S$.\\

(i)Let $N$ be a normal subgroup of $G$ of finite index such that $S \cap N \leq \Phi(S)$.  Then $\Phi^\infty_p(N)$ is a normal $p'$-Hall subgroup of $N$.\\

(ii)The quotient $G/O_{p'}(G)$ is virtually pro-$p$.\end{cor}

\begin{proof}(i) It is clear that $\Phi^1_p(SN) \cap S = \Phi^1_p(S)$, and hence $\Phi^\infty_p(SN) \cap S = \Phi^\infty_p(S) = 1$ by Tate's theorem.  It follows that $\Phi^\infty_p(SN)$ is a pro-$p'$ group, and so $\Phi^\infty_p(N)$ is also a pro-$p'$ group; clearly $\Phi^\infty_p(N)$ is normal in $N$.  But $N/\Phi^\infty_p(N)$ is a pro-$p$ group, so $\Phi^\infty_p(N)$ is a $p'$-Hall subgroup of $N$.\\

(ii)Since $d_p(G)$ is finite, there must be some open normal subgroup $N$ of $G$ such that $d_p(G/N) = d_p(G)$, so that $S \cap N \leq \Phi(S)$.  By part (i), $N/O_{p'}(N)$ is a pro-$p$ group.  Now $O_{p'}(N) \leq O_{p'}(G)$, so $G/O_{p'}(G)$ is an image of $G/O_{p'}(N)$, and $N$ has finite index in $G$, so $G/O_{p'}(N)$ is virtually pro-$p$.\end{proof}

\begin{cor}Let $G$ be a profinite group involving only finitely many primes, such that $d_p(G)$ is finite for every $p$.  Then $G$ is virtually pronilpotent.\end{cor}

\begin{proof}For each prime $p$ dividing $|G|$, we can choose an open normal subgroup $N_p$ of $G$ such that $N_p$ has a normal $p'$-Hall subgroup.  Set $N$ to be the intersection of these $N_p$, and note that $|G:N|$ is finite.  Then $N$ has a normal $p'$-Hall subgroup for every prime $p$, since $N \leq N_p$.  It follows that $N$ has a normal $\pi$-Hall subgroup for any set of primes $\pi$, and so $N$ is the direct product of its Sylow subgroups.  Hence $N$ is pronilpotent.\end{proof}

\section{The internal structure of the generalised Fitting subgroup}

\begin{thm}Let $G$ be a finite perfect group.  Then there is a finite group $\Gamma$, unique up to isomorphism, such that $\Gamma$ is a perfect central extension of $G$, and any finite perfect central extension $H$ is an image of $\Gamma$.  In particular, the order of any finite perfect central extension is at most $|\Gamma|$.\end{thm}

\begin{proof}See \cite{Suz}.\end{proof}

\begin{thm}\label{opilayer}Let $G$ be a profinite group.\\

(i) $O_\pi(G)$ is a characteristic pro-$\pi$ subgroup of $G$ which contains all subnormal pro-$\pi$ subgroups, and is the intersection over all open normal subgroups $N$ of the subgroups $O_N$ of $G$ such that $O_N/N = O_\pi(G/N)$.  In particular, $F(G)$ is pronilpotent and contains all pronilpotent subnormal subgroups, and is the intersection over all open normal subgroups $N$ of the subgroups $F_N$ of $G$ such that $F_N/N = F(G/N)$.\\
 
(ii) Let $Q$ be a component of $G$.  Then $Q$ is finite.\\
 
(iii)Any component $Q$ of $G$ commutes with both the Fitting subgroup and all components distinct from $Q$.  Hence $E(G)$ is an unrestricted central product of the components, and $F^*(G)$ is a central product of $F(G)$ and $E(G)$.\end{thm}

\begin{proof}We assume the finite case of the theorem, as this is well-known; see for example \cite{Suz} for a proof.\\ 

(i) Let $T$ be the intersection over all open normal subgroups $N$ of the subgroups $O_N$ of $G$ such that $O_N/N = O_\pi(G/N)$, and let $O=O_\pi(G)$.  By their construction, $O$ and $T$ are characteristic in $G$, and $T$ is a pro-$\pi$ group by the finite case of the theorem, so $T \leq O$.  For every $N \unlhd_o G$, $ON/N$ is generated by normal $\pi$-subgroups of $G/N$, so it is contained in $O_N/N$.  Hence $O \leq T$, and so $O=T$; in particular, $O$ is a pro-$\pi$ group.\\

Let $K$ be a normal subgroup of $G$.  Then $O_\pi(K)$ is characteristic in $K$, and so normal in $G$, and hence $O_\pi(K) \leq O_\pi(G)$.  It follows by induction on the degree of subnormality that $O$ contains $O_\pi(L)$ for all subnormal subgroups $L$ of $G$, and in particular it contains all subnormal pro-$\pi$ subgroups of $G$.\\

(ii) Since $Q$ is quasisimple, $Q/Z(Q)$ is a simple profinite group, and hence a finite simple group, since any infinite profinite group has a proper open normal subgroup.  It follows that all non-trivial finite images of $Q$ are isomorphic to perfect central extensions of the finite perfect group $Q/Z(Q)$, so $Q$ is an image of the maximal perfect central extension of $Q/Z(Q)$.  In particular, $Q$ is finite.\\
 
(iii) Let $Q$ be a component, and let $L$ be either a normal pro-$p$ subgroup of $G$, or some component distinct from $Q$.  Suppose that $[x,y] \not= 1$ for some $x \in Q, y \in L$.  Then $G$ has an open normal subgroup $N$ which intersects trivially with $Q$, which does not contain $[x,y]$, and such that $QN \not= LN$, and if $L$ is a component we can also choose $N$ to intersect trivially with $L$.  Then $G/N$ is a finite group, with a subgroup $LN/N$ that is either a normal $p$-subgroup of $G/N$ or a component of $G/N$ distinct from $QN/N$, and yet $LN/N$ does not commute with the component $QN/N$.  This contradicts the finite case of the theorem.\end{proof}

Say a profinite group $G$ is an $F^*$-group if $G = F^*(G)$.

\begin{cor}\label{opicor}(i) Let $G$ be a profinite group.  Then $G$ is an $F^*$-group if and only if $G/F(G)$ is perfect and $G/E(G)$ is pronilpotent.  In particular, every image of an $F^*$-group is an $F^*$-group.  Moreover, if $G$ is an $F^*$-group then $G/F(G)$ is a Cartesian product of finite simple groups.\\

(ii)Let $G$ be a profinite group, and let $H$ be a normal subgroup of $G$.  Then $E(H) = E(G) \cap H$ and $F(H) = F(G) \cap H$.  In particular, $H/E(H)$ is isomorphic to a normal subgroup of $G/E(G)$, and $H/F(H)$ is isomorphic to a normal subgroup of $G/F(G)$.\end{cor}

\begin{prop}Let $G$ be a profinite group, and let $M$ be a normal subgroup.  The following are equivalent:

(i) $M$ is an $F^*$-group;

(ii) $M \leq F^*(G)$;

(iii) $MN/N \leq F^*(G/N)$, for every open normal subgroup $N$ of $G$;

(iv) $MN/N$ is an $F^*$-group, for every open normal subgroup $N$ of $G$.\end{prop}

\begin{proof}Assume (i).  By Corollary \ref{opicor}, $F(M)$ is contained in $F(G)$ and $E(M)$ is contained in $E(G)$, so $M=E(M)F(M) \leq F^*(G)$.\\

Assume (ii).  Then $MN/N$ is an $F^*$-group that is a normal subgroup of $G/N$, so must be contained in $F^*(G/N)$ by the fact that (i) implies (ii).\\

Assume (iii).  Clearly $F^*(G/N)$ itself is a finite $F^*$-group.  It follows that $MN/N$ is also an $F^*$-group as $MN/N \unlhd F^*(G/N)$.\\

Assume (iv), and write $L(G)$ for the smallest subgroup of $G$ such that $G/L(G)$ is pronilpotent.  For all open normal subgroups $N$ of $G$, we see that $L(MN/N)$ is a central product of quasisimple groups, since $(MN/N)/E(MN/N)$ is pronilpotent.  It follows that $L(M)$ is a central product of quasisimple groups, by considering centralisers of finite images.  So $M/E(M)$ is pronilpotent.  Let $F_N$ be the preimage in $M$ of $F(MN/N)$.  Then $M/F_N$ is perfect; we see that the intersection of the $F_N$ is necessarily pronilpotent, so is contained in $F(M)$, which means $M/F(M)$ is perfect.  Hence $M$ is an $F^*$-group, by Corollary \ref{opicor}.\end{proof}

\section{A structure theorem for Fitting-degeneracy and Fitting-regularity}

First, we note some closure properties of $\FD$:

\begin{lem}The class $\FD$ is closed under subnormal subgroups, extensions, and sub-Cartesian products.\end{lem}

\begin{proof}For the first claim, it suffices to consider normal subgroups.  Let $N$ be a normal subgroup of $G \in \FD$.  Then $F^*(N)$ is contained in $F^*(G)$, which is trivial.\\
 
Now let $G$ be a profinite group with $N \lhd G$ such that $N, G/N \in \FD$.  Then $F^*(G) \cap N \leq F^*(N) =1$, and $F^*(G)N/N \leq F^*(G/N) =1$, so $F^*(G)=1$.\\

Now let $G$ be a profinite group which is a sub-Cartesian product of groups $H_i \in \FD$.  In other words, we have surjective maps $\rho_i$ from $G$ to $H_i$ for each $i$, such that the kernels of the $\rho_i$ have trivial intersection.  But $(F^*(G))^{\rho_i}$ is an $F^*$-group and hence trivial for each $i$, since it is normal in the Fitting-degenerate group $H_i$, so $F^*(G)$ must be trivial.\end{proof}

With these closure properties in mind we define the \emph{Fitting-degenerate residual} $O^\FD(G)$ of a profinite group $G$:
\[ O^\FD(G) := \bigcap \{ K \unlhd G \; |\; G/K \in \FD \} \]

\begin{cor}Let $G$ be a profinite group.  Then $G/O^\FD(G) \in \FD$, and $O^\FD(G) \in \FR$.\end{cor}

\begin{proof}The first claim follows immediately from the fact that $\FD$ is closed under sub-Cartesian products.  For the second,  let $G$ be a profinite group, let $R = O^\FD(G)$, and let $N = O^\FD(R)$.  Then $N$ is characteristic in $R$, and hence in $G$, and $G/N \in \FD$, since $\FD$ is closed under extensions.  Hence by definition $R \leq N$, so $N = R$, in other words $R$ has no non-trivial Fitting-degenerate images.\end{proof}

We now establish some closure properties of $\FR$.

\begin{lem}The class $\FR$ is closed under subnormal sections, subnormal joins and extensions.\end{lem}

\begin{proof}For the first claim, it suffices to show $\FR$ is closed under normal subgroups, by induction on the degree of subnormality and by the definition of Fitting-regularity.  Let $G \in \FR$, and suppose $N \lhd G$ such that $O^\FD(N) < N$.  Then $O^\FD(N)$ is a normal subgroup of $G$, since it is characteristic in $N$, so by replacing $G$ by $G/O^\FD(N)$ and $N$ by $N/O^\FD(N)$, we may assume that $N$ is Fitting-degenerate.  Since $N$ is Fitting-degenerate it has trivial centre, and so the normal subgroup $H=C_G(N)$ of $G$ has trivial intersection with $N$.  Hence $N$ is isomorphic to $NH/H$, and the centraliser in $G/H$ of $NH/H$ is trivial.  In particular, $NH/H$ intersects non-trivially with any non-trivial normal subgroup of $G/H$.  But then $F^*(G/H) \cap NH/H$ is a non-trivial normal subgroup of $NH/H$ that is an $F^*$-group, contradicting the Fitting-degeneracy of $NH/H$.\\

Now let $G$ be a profinite group generated by subnormal subgroups $N_i \in \FR$.  Let $M$ be a proper normal subgroup of $G$.  Then there is some $N_i$ not contained in $M$, and so $G/M$ has a subnormal subgroup $N_i M/M$ which is isomorphic to the non-trivial image $N_i/(N_i \cap M)$ of $N$, so $F^*(G/M) \geq F^*(N_i M/M) > 1$.  As $M$ was an arbitrary proper normal subgroup, this means $G \in \FR$.\\

Now let $G$ be a profinite group with $N \lhd G$ such that $N,G/N \in \FR$.  Let $M$ be a proper normal subgroup of $G$.  By the above argument, it suffices to show that $G/M$ is not Fitting-degenerate in the case that $N \leq M$.  This follows from the fact that $G/M$ is an image of $G/N$ in this case.\end{proof}

\paragraph{Remark}All countably based profinite groups are subgroups of the Fitting-regular group $\prod_{n \geq 5} \Alt(n)$ (see \cite{Wil}), but there are non-trivial Fitting-degenerate countably based profinite groups.  Hence the class $\FR$ is not closed under subgroups, though it is closed under taking open subgroups, as will be seen shortly.  It is not clear whether or not a subgroup of a Fitting-regular prosoluble group can fail to be Fitting-regular.\\

In light of the above closure properties, we define the \emph{Fitting-regular radical} $O_\FR(G)$ of $G$ to be the group generated by all normal Fitting-regular subgroups of $G$.\\

We are now ready to prove the desired structure theorem:

\begin{thma}Let $G$ be a profinite group.  Then $G$ has a characteristic subgroup $R = O^\FD(G) = O_\FR(G)$, such that:

(i) A subnormal subgroup of $G$ is Fitting-regular if and only if it is contained in $R$;

(ii) $G/R$ is Fitting-degenerate, and covers every Fitting-degenerate quotient of $G$.\\

We also have $O_\FR(H) = O_\FR(G) \cap H$, given any open subgroup $H$ of $G$.
\end{thma}

\begin{proof}Let $R=O^\FD(G)$.  We have already seen that $R$ is Fitting-regular, so $R \leq O_\FR(G)$, and that this implies all subnormal subgroups of $R$ are Fitting-regular.\\

Conversely, let $N$ be a Fitting-regular subnormal subgroup of $G$.  Then $NR/R$ is Fitting-regular, as it is isomorphic to an image of $N$, but it is also Fitting-degenerate, as it is a subnormal subgroup of $G/R \in \FD$.  Hence $NR/R=1$, and so $N \leq R$.  This demonstrates that (i) holds, and also that $R = O_\FR(G)$.\\

We have also already seen that $G/R$ is Fitting-degenerate, and by definition it covers every Fitting-degenerate quotient of $G$.  This is (ii).\\

Finally, let $H$ be an open subgroup of $G$.  Let $K=O_\FR(G)$.  Then $G/K$ is Fitting-degenerate, so the core of $HK/K$ in $G/K$ is Fitting-degenerate; this means that $HK/K$ has a Fitting-degenerate normal subgroup of finite index, and hence $O_\FR(HK/K)$ must be finite.  But then the elements of $O_\FR(HK/K)$ have centralisers of finite index in $HK/K$, and hence also in $G/K$, and they are of finite order; thus $O_\FR(HK/K)$ is contained in a finite normal subgroup of $G/K$, by Dicman's lemma.  Since $G/K$ is Fitting-degenerate, this implies $O_\FR(HK/K)=1$, and so $HK/K$ is Fitting-degenerate.  Hence $H/(H \cap K)$ is a Fitting-degenerate image of $H$, which ensures that $O_\FR(H) \leq H \cap K$.\\

Let $M$ be the core of $H$ in $G$.  Then $M \cap K$ is an open normal subgroup of $H \cap K$, and $M \cap K$ is also a normal subgroup of $K$ and hence Fitting-regular.  It follows that $H \cap K$ itself is (Fitting-regular)-by-finite, and thus Fitting-regular, as $\FR$ is closed under extensions.  So $H \cap K \leq O_\FR(H)$.\end{proof}

\section{The generalised Fitting subgroup of a Fitting-regular group}

\begin{thmb}Let $G \in \FR$.  Then $C_G (F^*(G)) = Z(F(G))$.\end{thmb}
 
\begin{proof}We consider the subgroup $H = C_G(F(G))$ and its Fitting subgroup.  Since $H$ is normal in $G$, we have $F(H) \leq F(G)$ and hence $F(H) = Z(H) = Z(F(G))$.\\
 
Set $K=H/Z(H)$.  Let $L$ be the subgroup of $H$ such that $L/Z(H) = F(K)$.  Then $L$ is a central extension of $F(K)$, and hence pronilpotent; it is also normal in $H$.  Hence $L \leq F(H)=Z(H)$, and so $F(K)=1$.  Now consider $D = C_K(E(K))$.  Then $D$ cannot contain a component, as this would be contained in $E(K)$, and yet a component cannot centralise itself; so $E(D)=1$.  Also, $F(D) \leq F(K) = 1$, so $F^*(D)=1$.  However, $D$ is a normal section of $G$, so $D$ is Fitting-regular.  Hence $D=1$, in other words $K$ acts faithfully on $E(K)$.\\

Let $T$ be a component of $K$.  Then $T$ is simple as $K$ has no non-trivial subnormal abelian subgroups.  Let $U$ be the subgroup of $H$ such that $U/Z(H) = T$.  Then $U$ is a central extension of $T$, and the intersection $V$ of all derived subgroups of $U$ is isomorphic to a perfect central extension of $T$.  We see that $VZ(H)/Z(H) = T$ here, and $V$ is subnormal in $G$ by construction, so $V \leq E(G)$.  This proves that $E(G)Z(H)/Z(H) \geq E(K)$.\\

We conclude that $H/Z(H)$ acts faithfully on $E(G)Z(H)/Z(H)$, so that $C_H(E(G)) \leq Z(H)=Z(F(G))$.  But $C_H(E(G))$ is precisely $C_G(E(G)) \cap C_G(F(G))$, which is the centraliser of $E(G)F(G) = F^*(G)$.\end{proof}

\paragraph{Remark}The converse of this theorem is false: consider for instance a profinite group of the form $G = V:L$ where $V$ is elementary abelian of countably infinite rank, and $L$ acts faithfully on $V$.  Since $V$ admits a faithful action of any countably based profinite group, we can choose $L$ to be non-trivial Fitting-degenerate, and yet $F(G) \geq V \geq Z(F(G))$.

\section{A sufficient condition for Fitting-regularity}

We note some properties of the class of Sylow-finite groups.

\begin{prop}(i) The class of Sylow-finite groups is closed under subgroups, quotients and finite direct products.\\

(ii) Every Sylow-finite group is Fitting-regular.\end{prop}

\begin{proof}(i) It is clear that the properties of having finite Sylow subgroups, and having Sylow subgroups with centralisers of finite index, are both preserved by the operations in question.\\
 
(ii) Let $G$ be Sylow-finite and let $K$ be a non-trivial image of $G$.  Then $K$ is Sylow-finite and has a non-trivial Sylow subgroup, and hence a non-trivial finite normal subgroup $N$, which is then necessarily contained in $O_\FR(K)$.   Hence $K$ cannot be Fitting-degenerate.\end{proof}

The following is a simple but useful observation.

\begin{lem}\label{xstar}Let $G$ be a profinite group, and let $N$ be an open normal subgroup of $G$.  Then $d_*(G)$ is finite if and only if $d_*(N)$ is finite, and $c_*(G)$ is finite if and only if $c_*(N)$ is finite.\end{lem}

\begin{proof}Since $G/N$ is finite, it must be a $\pi$-group for some finite set of primes $\pi$.  Suppose $x_*(G)$ is finite, where $x$ is either $c$ or $d$.  Given $p \in \pi'$, the $p$-Sylow subgroups of $N$ are the same as those of $G$, and so $d_p(N) = d_p(G) \leq d_*(G)$ and $c_p(N) = c_p(G) \leq c_*(G)$.  On the other hand for $p \in \pi$, a $p$-Sylow subgroup $P$ of $N$ extends to a $p$-Sylow subgroup $S$ of $G$, with $|S/P|$ finite, and so $c(P)$ and $d(P)$ are finite.  Since $\pi$ is finite, we conclude that $x_*(G)$ is finite.\\
 
Conversely, suppose $x_*(N)$ is finite.  Then $x_p(G) \leq x_*(N)$ for all $p \in \pi'$, so for $x_*(G)$ to be infinite, we would need $x_p(G)$ infinite for some $p$.  But a Sylow subgroup $S$ of $G$ is a finite extension of a Sylow subgroup of $N$, all of which are finitely generated; so $d(S)$ is finite, and hence $c(S)$ is also finite.\end{proof}

The following lemma is the main part of the proof of the theorem.  We denote by $\lambda_n$ the set of primes at most $n$.

\begin{lem}\label{sylact}Let $G$ be a prosoluble group with $c_*(G) =c$, for some integer $c$.  Let $K$ be the smallest normal subgroup of $G$ such that $G/K$ has exponent dividing $\eb(c)$ and derived length at most $\db(c)$.  Then
\[ K' \leq S[G,S]C_G(S) \]
for any Sylow subgroup $S$ of $G$.\\
 
Moreover any $p$-Sylow subgroup of $K'$ centralises a $\lambda'_n$-Hall subgroup of $G$, for some $n$ depending on $p$ and $G$.\end{lem}

\begin{proof}Let $S$ be a $p$-Sylow subgroup of $G$, for some prime $p$.  For the first assertion, it suffices to show that $G/(S[G,S]C_G(S))$ has an abelian normal subgroup $K/S[G,S]C_G(S)$ such that the quotient $G/K$ has exponent dividing $\eb(c)$ and derived length at most $\db(c)$.  But by the Frattini argument as applied to Sylow's theorem, we have
\[ G = [G,S]N_G(S) \]
so that $G/(S[G,S]C_G(S))$ is isomorphic to a quotient of $L = N_G(S)/SC_G(S)$.  Now $N_G(S)/C_G(S)$ acts faithfully on $S$, and so by coprime action $L$ acts faithfully on $S/\Phi(S)$.  By Corollary \ref{malcor}, $L$ therefore has an abelian subgroup $M$ such that $L/M$ has exponent dividing $\eb(c)$ and derived length at most $\db(c)$, and so the same must be true for $G/(S[G,S]C_G(S))$.\\

Now consider the action of $G$ on $O_{p'}(G)$.  First of all, $G/O_{p'}(G)$ is virtually pro-$p$, and hence pro-$\lambda_n$ for some $n$.  For all primes $q > n$, we have some $q$-Sylow subgroup $T$ of $G$ contained in $O_{p'}(G)$.  Applying the previous part, we get
\[ K' \leq T[G,T]C_G(T) \leq O_{p'}(G)C_G(T) \]
from which it follows that a $p$-Sylow subgroup of $K'$ is contained in $C_G(T)$.  So the centraliser of some $p$-Sylow subgroup $R$ of $K'$ contains a $q$-Sylow subgroup, and hence this is true for any $p$-Sylow subgroup by Sylow's theorem.  So $C_G(R)$ contains a $q$-Sylow subgroup of $G$ for all primes $q > n$, and hence contains a $\lambda'_n$-Hall subgroup of $G$, since $C_G(R)$ is prosoluble.\end{proof}

\begin{thmc}Let $G$ be a profinite group such that $c_*(G)$ is finite.  Then $G$ is virtually pronilpotent-by-(Sylow-finite)-by-abelian.\end{thmc}

\begin{proof}Since $c_*(G)$ is finite, $c_2(G)$ is finite and hence $d_2(G)$ is finite.  So $G/O_{2'}(G)$ is virtually pro-$2$, in other words it has a $2$-core $H/O_{2'}(G)$ of finite index.  By the Odd Order Theorem (\cite{Fei}), $O_{2'}(G)$ is prosoluble, so $H$ is also prosoluble.  Lemma \ref{xstar} ensures that $c' = c_*(H)$ is finite.  Now applying Lemma \ref{sylact} to $H$, we obtain a characteristic subgroup $K$ of $H$ such that $H/K$ has exponent dividing $\eb(c')$ and derived length at most $\db(c')$, and any $p$-Sylow subgroup of $K'$ centralises a $\lambda'_n$-Hall subgroup of $H$, for some $n$ depending on $p$ and $H$.  Given a $p$-Sylow subgroup $R$ of $H/K$, we see that $R$ is finitely generated since $d_p(H)$ is finite, and also $R$ has exponent dividing $\eb(c')$ and derived length at most $\db(c')$.  Hence $R$ is finite, of order bounded by a function of $d_p(H)$ and $c'$.  Since $H/K$ has finite exponent, $|H/K|$ involves only finitely many primes.  Hence $H/K$ is finite.  Since $K$ is characteristic in $H$, it is normal in $G$, and we have $G/K$ finite.\\

It remains to show that $K'/F(K')$ is Sylow-finite.  Let $S$ be a $p$-Sylow subgroup of $K'$, so that $S$ centralises a $\lambda'_n$-Hall subgroup $L$ of $H$ for some $n$.  We may choose $n$ to be greater than $p$, so that $S$ is contained in a $\lambda_n$-Hall subgroup $M$ of $H$.  Now $M$ is virtually pronilpotent, so $O_p(M) \cap K'$ has finite index in $S$, and $O_p(M)$ is centralised by $L$ and normalised by $M$, so it is normal in $LM=H$ and hence $O_p(M) \cap K'$ is contained in $F(K')$.  This establishes that all Sylow subgroups of $K'/F(K')$ are finite.  But this means that a $\lambda'_n$-Hall subgroup of $K'/F(K')$ has finite index, for any integer $n$, and so every Sylow subgroup also has a centraliser of finite index.\end{proof}

\paragraph{Remark}We use the Odd Order Theorem here to show that $O_{2'}(G)$ is prosoluble, but for the rest of the proof it suffices to show (or include as a hypothesis) that $G$ is virtually prosoluble, for which it is sufficient to establish that there is some integer $n$ such that all finite $\lambda'_n$-groups are soluble.  It may be possible to prove this last claim for a sufficiently large $n$ in a way that bypasses large parts of the proof of the Odd Order Theorem.

\begin{corn}Let $G$ be a profinite group such that $c_*(G)$ is finite, and let $H$ be a subgroup of $G$.  Then $H$ is Fitting-regular.\end{corn}

\begin{proof}Since $G$ is virtually pronilpotent-by-(Sylow-finite)-by-abelian, the same is true of $H$.  The property of being Fitting-regular is closed under extensions, so the result follows from the fact that pronilpotent, Sylow-finite, abelian and finite groups are all Fitting-regular.\end{proof}

We have established a sufficient condition for Fitting-regularity.  On the other hand, it is easy to imagine non-trivial profinite groups which are Fitting-degenerate if we impose some alternative finiteness conditions on the Sylow structure.  These are best understood in general terms rather than by giving specific constructions.\\

Suppose we have a sequence of finite groups $G_i$ and primes $p_i$, with the following properties:

(i) $G_1$ is cyclic of order $p_1$;

(ii) $G_{i+1} = V_{i+1}:G_i$, where $V_{i+1}$ is an elementary abelian $p_i$-group on which $G_i$ acts faithfully.\\

It is easy to see that, no matter what finite group we have for $G_i$, it is possible to choose a suitable $G_{i+1}$, and moreover we have a free choice of the primes $p_i$.  These $G_i$ also form an inverse system in an obvious way, giving rise to a prosoluble group $G$.\\

Now consider the (generalised) Fitting subgroup of $G$ (note $F^*(G)=F(G)$ in this case).  For each $G_i$, the image of $F(G)$ in $G_i$ is contained in $F(G_i)$.  But it is clear from the construction of $G_i$ that $F(G_i)$ is always a $p_i$-group.  Suppose $F(G)$ is non-trivial; then for some $p$ it contains a non-trivial pro-$p$ element $x$.  This means that, for all sufficiently large $i$, the image of $x$ in $G_i$ is a non-trivial $p$-element of $F(G_i)$.  Hence $p_i=p$ for all sufficiently large $i$.\\

By choosing the $p_i$ appropriately, we can thus ensure $G$ is Fitting-degenerate.  One such choice of $p_i$ would be to make $p_i = p$ for $i$ odd and $p_i=q$ for $i$ even, for some pair of distinct primes $p$ and $q$.  Another choice of $p_i$ would be to make all the $p_i$ distinct, so that the set of $p_i$ could be any infinite set of primes, and every Sylow subgroup of $G$ is finite and elementary abelian.\\

The above arguments show the following:

\begin{propn}(i) There exist non-trivial Fitting-degenerate prosoluble groups, all of whose Sylow subgroups are finite elementary abelian groups.\\
 
(ii) There exist non-trivial Fitting-degenerate prosoluble groups which are countably based pro-$\{p,q\}$ groups, for any distinct primes $p$ and $q$.\end{propn}

\section{Fitting-degeneracy and countably-based images}

In this final section, we give a couple of results to show that the property of Fitting-degeneracy can be investigated it terms of `small' images.\\

We say a finite group $G$ is \emph{primitive} if there exists a maximal subgroup $H$, such that the core of $H$ in $G$ is trivial.  Note that a finite image $G/N$ of a profinite group $G$ is primitive if and only if $N$ is the core of a maximal open subgroup of $G$.

\begin{thmd}Let $G$ be a non-trivial profinite group.  Then $G \in \FD$ if and only if the following holds, for any $x \in G \setminus 1$:
 
$(*)$ There is an open normal subgroup $K$ of $G$, depending on $x$, such that $G/K$ is primitive and $xK$ is not contained in $F^*(G/K)$.\end{thmd}

\begin{proof}Suppose $G$ is not Fitting-degenerate, and take $x \in F^*(G) \setminus 1$.  Then $xK \in F^*(G/K)$ given any $K$, so $(*)$ is false.\\

Now assume $G$ is Fitting-degenerate.  This means that $G$ has no non-trivial pronilpotent normal subgroups, so in particular $\Phi(G)$ is trivial.  Hence for any element $x$ of $G \setminus 1$, we can find a maximal subgroup $H$ such that $H$ does not contain $x$, and so the core $K$ of $H$ also does not contain $x$.  Since $G/K$ is primitive, this shows that $G$ is the inverse limit of finite primitive groups.  If we had $xK \in F^*(G/K)$ for a given $x$ in every such image, it would imply $x \in F^*(G)$, contradicting the assumption that $G$ is Fitting-degenerate.  This proves $(*)$.\end{proof}

\begin{thme}Let $G$ be a profinite group that is not Fitting-regular.  Then there is a normal subgroup $N$ of $G$ such that $G/N$ is countably based and Fitting-degenerate.\end{thme}

\begin{proof}Without loss of generality, we may replace $G$ by $G/O^\FD(G)$, and so assume $G$ is Fitting-degenerate.  Set $N_1$ to be any proper open normal subgroup of $G$.  We obtain open normal subgroups $N_{i+1}$ for $i \in \mathbb{N}$ inductively as follows:\\

Let $M$ be the preimage of a non-trivial normal subgroup of $G/N_i$.  Suppose for every open normal subgroup $K$ of $G$ contained in $N_i$ that $F^*(G/K)$ covers $M/N_i$.  Then by a standard inverse limit argument, we would obtain a subgroup of $F^*(G)$ covering $M/N_i$, which is impossible as $G$ is Fitting-degenerate.  So there must be an open normal subgroup $K_M \leq N_i$ of $G$ such that $F^*(G/K_M)$ does not cover $M/N_i$.  Set $N_{i+1}$ to be the intersection of all the chosen normal subgroups $K_M$.\\

Now set $N$ to be the intersection of all the $N_i$; by construction, $G/N$ is countably based.  Also by construction, $F^*(G/N_{i+1})$ cannot cover any non-trivial normal subgroup of $G/N_i$, and hence $F^*(G/N_{i+1}) \leq N_i/N_{i+1}$.  As a result $F^*(G/N) \leq N_i/N$ for all $i$, so $G/N$ is Fitting-degenerate.\end{proof}

\section{Acknowledgments}
This paper is based on results obtained by the author while under the supervision of Robert Wilson at Queen Mary, University of London (QMUL).  The author would also like to thank Charles Leedham-Green for introducing the author to the study of profinite groups and for his continuing advice and guidance in this area.  The author acknowledges financial support provided by EPSRC and QMUL for the duration of his doctoral studies.

\end{document}